\documentclass[11pt]{amsart}
\usepackage{amsmath}
\usepackage{amsthm}
\usepackage{amssymb}

\newtheorem{theorem}{Theorem}

\newtheorem{proposition}[theorem]{Proposition}
\newtheorem{lemma}[theorem]{Lemma}
\newtheorem{corollary}[theorem]{Corollary}

\theoremstyle{remark}
\newtheorem*{remark}{Remark}
\newtheorem*{acknowledgements}{Acknowledgements}
\begin{document}

\title[Eigensystem of an $L^2$-perturbed harmonic oscillator]
{Eigensystem of an $L^2$-perturbed harmonic oscillator is an unconditional basis}

\author{James Adduci}

\address{Department of Mathematics,
The Ohio State University,
 231 West 18th Ave,
Columbus, OH 43210, USA}

\email{adducij@math.ohio-state.edu}

\author{Boris Mityagin}

\address{Department of Mathematics,
The Ohio State University,
 231 West 18th Ave,
Columbus, OH 43210, USA}

\email{mityagin.1@osu.edu}


\begin{abstract}
We prove the following. For any complex 
valued 
$L^p$-function $b(x)$, 
$2 \leq  p < \infty$ or $L^\infty$-function with the norm $\| b | L^{\infty}\| < 1$, 
the spectrum of a perturbed harmonic oscillator 
operator $L = -d^2/dx^2 + x^2 + b(x)$ in 
$L^2(\mathbb{R}^1)$ is discrete and eventually simple. 
Its SEAF (system of eigen- and 
associated functions) is an unconditional basis in $L^2(\mathbb{R})$.
\end{abstract}


\maketitle

\section{Introduction} \label{INTRO}

In this paper we condsider the perturbed operator 
\begin{align} 
\label{1A}
L=L^0 + B
\end{align}
where
\begin{align*} 
L^0 := -\frac{d^2}{dx^2} + x^2 
\end{align*}
is a harmonic oscillator
and $B$ is multiplication by a complex-valued function, 
\begin{align*} 
B \phi(x) = b(x) \phi(x)
\end{align*}
or maybe a more general linear operator.
The spectrum of $L^0$ is the set 
\begin{align*} 
\text{Sp}(L^0) = \{\lambda^0_k = 2k+1: k \in \mathbb{Z}_+=0,1,2,\ldots \}.
\end{align*}
The corresponding eigenfunctions are the Hermite functions $h_k(x), k \in \mathbb{Z}_+$ 
(see for example \cite[Ch. 6, Sect 34]{dirac}), \cite[Ch. 5, Sect. 4]{levitan}; $\|h_k\|_2=1$.
Define the Banach space 
\begin{align} 
\label{5A}
V = \{ \phi \in L^2_{\text{loc}}(\mathbb{R}): \| \phi h_k \|_2 < \infty \quad \forall k \in \mathbb{Z}_+ 
 \quad \text{and} \quad \lim_{k \rightarrow \infty} \| \phi h_k \|_2 =0 \}, 
\end{align}
with the norm
\begin{align}
\label{6A}
\quad \| \phi \| &= \sup \{ \|\phi h_k \|_2 \}.
\end{align}
Our proofs in this paper use in an essential way the condition  $b \in V$. 
In Section \ref{1E} we will use known estimates for Hermite functions 
to prove that the following spaces are embedded in $V$:
\begin{align}
\label{Oa}
L(p,\alpha)&= \{ \phi: (1+|x|^2)^{\alpha/2} |\phi(x)|  \in L^p(\mathbb{R}) \}, \,\, \alpha/2 + t(p) \leq 0  
\end{align}
where $t$ is defined in (\ref{0305_4}) and
\begin{align}
\label{Ob}
L^\infty_0(\mathbb{R})  &= 
 \{ \phi \in  L^\infty(\mathbb{R}): \text{ess sup}_{|s| \geq t } |\phi(s)| \rightarrow  0 \quad \text{as} \quad t \rightarrow \infty \}.
\end{align} 
We can claim as a special case
of (\ref{Oa}) that the spaces $L^p(\mathbb{R}), \quad 2\leq p < \infty$ are embedded in $V$  - see Lemma \ref{0305_5} and Proposition 
\ref{0312_1}.

We will now state our main results. 
Their proofs are given in Sections \ref{1C} and \ref{1D}.
Put 
\begin{align*} 
\Pi(a,b) := \{x+iy \in \mathbb{C}: |x| < a, |y| < b \},
\end{align*} 
\begin{align*} 
D(p,r) := \{z \in \mathbb{C}: |z-p| < r \}
\end{align*}
and 
\begin{align}
\label{9A}
S(n) := \Pi(2n,Y) \cup \left( \bigcup_{k=n}^{\infty} D(\lambda^0_k, 1/16) \right)
\end{align}
with 
\begin{align}
\label{10A}
Y= 8  \left( \|b \| +2 \pi  \|b\|^2 \right). 
\end{align}
For $z \notin \text{Sp}(L)$ (resp. $z \notin \text{Sp}(L^0)$) put 
$R(z):= (z-L)^{-1}$ (resp. $R^0(z):= (z-L^0)^{-1}$).
\begin{proposition} 
\label{11A}
Suppose $b \in V$. The spectrum of $L$ is discrete and 
there exists $N_* \in \mathbb{N}$ such that
\begin{align}
\label{12A}
\text{Sp}(L) \subset S:= S(N_*)
\end{align}
and for each 
$n \geq N_*$ the disk $D(\lambda^0_n, 1/16)$ contains exactly 
one  eigenvalue $\lambda_n$ of $L$.
\end{proposition}
Put
\begin{align} 
\label{13A}
S_*   &:= \frac{1}{2 \pi i} \int_{\partial \Pi(2N_*,Y)} R(z) dz,    \\
\label{14A}
P_k^0 &:=\frac{1}{2 \pi i} \int_{\partial D(\lambda^0_k,1)} R^0(z) dz, \quad \text{for $k \in \mathbb{Z}_+$} 
\end{align}
and 
\begin{align} 
\label{15A}
P_k := \frac{1}{2 \pi i} \int_{\partial D(\lambda^0_k,1)} R(z) dz, \quad \text{for $k\geq N_*$}.   
\end{align}
By Proposition \ref{11A} all the integrals in (\ref{13A}-\ref{15A}) are well defined.
Of course, $P_k^0 f = \left< f, h_k \right> h_k \quad \forall k$ and $\dim P_k^0 = 1$.
\begin{proposition}
\label{11AA}
The constant $N_*$ from Proposition \ref{11A} can be chosen in such a way that 
\begin{align}
\label{16A}
\text{dim}(P_k) &= \text{dim}(P_k^0)=1 \quad \forall k\geq N_*
\end{align}
and
\begin{align}
\label{18A}
\quad \text{dim}(S_*)&=\text{dim}(S_*^0) = N_* \quad 
\text{where} 
\quad S^0_* = \sum_{0}^{N_*-1} P_j^0.
\end{align}
Also, 
\begin{align} 
\label{19A}
\|R(z)\| \leq 32 \quad \forall z \notin S \in (\ref{12A}), 
\end{align}
\begin{align} 
\label{20A}
\|P_n \| \leq 32, 
\end{align}
and 
\begin{align} 
\label{21A}
\left\| \frac{1}{2\pi i} 
\int_{|z- \lambda_n^0| = 1} \frac{R(z)}{z - \lambda_n} dz \right\| 
\leq  35
\end{align}
whenever 
$n \geq N_*$.
\end{proposition}

Let us notice that for us Propositions \ref{11A} and \ref{11AA} have a limited purpose; they
give an accurate construction of spectral projections and the system
of eigenfunctions and associated functions (SEAF) of the operator $L \in (\ref{1A})$
so we can talk about spectral decompositions (\ref{23A})
in our main theorem -- Theorem \ref{22A}.

More deliberate analysis would give  asymptotics of $L$'s eigenvalues $(\lambda_k)$.
Such asymptotics - at least for real-valued $b(x)$ - could be found in \cite{akhmerova}
(see references there as well). Another group of questions - inverse problems  - for real
valued $b(x)$ such that $b'(x)$ and $xb(x)$ are in $L^2$ as well is considered in a 
series of papers \cite{chelkak}-\cite{chelkak4}.

\begin{theorem}
\label{22A}
Suppose $b \in V$; the operator (\ref{1A}) 
generates the spectral decompositions
\begin{align} 
\label{23A}
f= S_*f + \sum_{k\geq N_*} P_kf \quad 
\text{for all} \,\,\, f \in L^2(\mathbb{R})
\end{align}
where $S_*, P_k$ are defined by (\ref{13A}, \ref{15A}) 
and $N_*$ is from Proposition \ref{11A}. 
These series converge unconditionally.
\end{theorem}

Equation (\ref{23A}) could be written as
\begin{align} 
\label{24A}
f= S_*f + \sum_{k\geq N_*} \left<f,\psi_k \right> \varphi_k \quad 
\text{for all} \quad f \in L^2(\mathbb{R})
\end{align}
since $P_k$ are $1$-dimensional projections for 
$k  \geq N_*$. In 
(\ref{3C})
we give conditions for 
$N_*$ which guarantee 
(\ref{16A}).
$S_*$ is an $N_*$-dimensional projection and 
$E_*= \text{Image}(S_*)$ is an invariant subspace for $L$,
but we cannot say more about the structure of 
$L|E_*$. It is likely that $L|E_*$ has Jordan subspaces,
and $\mu \in \text{Sp}(L|E_*)$  
could have any algebraic multiplicity 
$m \geq 2$ but of course its 
geometric multiplicity $ \leq 2$ 
(really, $ \leq 1$ as an elementary 
analysis of the Wronskian shows). 
It would be interesting to get 
an analog of V. Tkachenko's results for 
\begin{align}
\label{0319_1}
M&=M_0 + V \\
\notag M_0 &= -d^2/dx^2,\quad  Vf = v(x)f(x), \,\,\, v \in  L^2(I)
\end{align}
on a finite interval $I=[0,\pi].$ 
In \cite[Proposition 1,2]{tkachenko}  it 
is shown that for every finite set 
$K = \{k_1, \ldots, k_t\}$ 
of pairwise distinct points in $\mathbb{C}$ and 
every set of $Q(K) = \{ q_1, \ldots, q_t\}$ 
of positive integers there exists a 
Sturm-Liouville operator (\ref{0319_1}) 
for which $\{ k_1, \ldots, k_t \}$ 
are points of the Dirichlet spectrum 
(or of the periodic spectrum, 
or of the anti-periodic spectrum)
with algebraic multiplicities equal 
to respective numbers from $Q(K)$.

Many questions about properties of SEAF of operators 
$L \in (\ref{1A})$ with $b(x) = u(x) + iv(x)$,
$u(x)=u(-x)$, $v(-x)=-v(-x)$ have been 
raised in the context of the PT-operator theory
(see \cite{bender2}, \cite{znojil}, \cite{bender},  \cite{mostafazadeh}, \cite{caliceti} ). 
In \cite{albeverio} it is shown (Section 6 and Theorem 5.8)
that if $u(x)=0$ and $|v(x)|\leq M < 2/\pi$ 
then $L$ is similar to a self-adjoint 
operator with discrete spectrum.
It happens this is true even with $M <1$.

In Section \ref{NEWNEWONE1} we consider bounded potentials $b(x) \in L^\infty$
or, more generally, perturbations in (\ref{1A}) with $B$ being a bounded operator in $L^2(\mathbb{R})$. 
V. Katsnelson's approach \cite{katsnelson}-\cite{khrushchev} to analysis of SEAF of dissipative operators leads 
to claims (see Theorem \ref{NEWNEWONE2}) that such a perturbation is `` good", i.e., an 
analog of Theorem \ref{22A} holds if $\|B \| < 1$. The constant $1$ 
is sharp as the  special (counter) examples in Section \ref{NEWNEWONE1} show.

In the multi-dimensional case see M. Agranovich's surveys 
\cite{agranovich1}, \cite[Ch. 5]{agranovich}  
on basis properties of eigensystems of weak perturbations 
of self-adjoint elliptic pseudodifferential operators 
on closed manifolds. In particular
the results of \cite[Sect. 6.2]{agranovich1}.


\section{Technical preliminaries: the discrete Hilbert transform} \label{1B}

\subsection{} \label{1H}
Let $G$ be the discrete Hilbert transform
\begin{align}\label{2H}
(G \xi)_n = \sum_{\substack{k=0\\ k\neq n}}^{\infty} \frac{\xi_k}{k-n}, \quad \text{for} \quad \xi \in \ell^2(\mathbb{Z}_+).
\end{align}
\emph{The operator $G$ is a bounded mapping from  $\ell_2$ to $\ell_2$} 
- see for example \cite[Sect. 8.12, statement 294]{hlp}, and \cite[Ch. 4, Thm. 9.18]{zygmund}. 

Given a positive weight sequence $\{W(k) \}_{k \geq 0}$ define 
\begin{align} \label{0810_1} 
\ell^2(W) = \{ \xi: \sum_{0}^{\infty} |\xi_k|^2 W(k) < \infty \}
\end{align}
and denote $w(k) = W(k)^{-1}$. We will use this convention throughout the paper.
\emph{If $W(k) = (k+1)^{\alpha}$ with $\alpha < 1$ then 
$G$ is a bounded mapping from $\ell^2(W)$ into itself}. That is,
\begin{align} \label{0303_1}
\sum_{n=0}^{\infty} (n+1)^{\alpha} \left| \sum_{k\neq n} \frac{\xi_k}{k-n}\right|^2
\leq C \sum_{n=0}^{\infty} (n+1)^{\alpha} |\xi_n|^2.
\end{align}
Furthermore, \emph{given any weight sequence $\psi(k) \rightarrow \infty$
there exists another 
weight sequence} 
\begin{align} \label{0303_3}
W_{\psi}(k) \uparrow \infty \quad \text{with} \quad  
W_{\psi}(k) \leq  \psi(k) 
\end{align}
(so that $\ell^2(\psi) \subset \ell^2(W_{\psi})$)
\emph{such that 
$G$ is a bounded mapping from $\ell^2(W_{\psi})$ into itself,
that is,} 
\begin{align} \label{0303_2}
\sum_{n=0}^{\infty} W_{\psi}(n) \left| \sum_{k\neq n} \frac{\xi_k}{k-n}\right|^2
\leq C \sum_{n=0}^{\infty} W_\psi(n) |\xi_n|^2.
\end{align}
For a proof of these facts we refer to Appendix, Corollary \ref{0304_1}.
\subsection{} \label{23HA}
Define a perturbed Hilbert transform $G_\tau$ by
\begin{align} 
\label{24H}
(G_\tau \xi)_n = 
\sum_{\substack{k=0 \\k\neq n}}^{\infty} 
\frac{\xi_k}{k+\tau_k - n}.
\end{align}
An analog of inequalities (\ref{0303_1}),(\ref{0303_2}) still holds for small $\tau = (\tau_k)_{k=0}^{\infty}$. 
\begin{lemma} \label{25H}
Suppose $W$ is a positive sequence and the following hold: 
\\ (a) \quad  $G \in (\ref{2H})$ is a bounded map from $\ell^2(W)$ to $\ell^2(W)$,  
\\(b) \quad $\tau$ is a sequence such that $|\tau_k| \leq  1/16$,
\\
(c)
\begin{align}
\label{25HA}
\sum_{n=0}^{\infty} \frac{r(n)}{(1+n)^2} < \infty \quad \text{where} \quad r(n) = \sup \left\{ W(i+n)w(i): i\geq -n \right\}.
\end{align}
Then $G_{\tau}$ is a bounded map from $\ell^2(W)$ to $\ell^2(W)$.
\end{lemma}
\begin{proof}
With the boundedness of $G$ , it suffices  to show that the difference
$G-G_{\tau}$
is a bounded map from $\ell^2(W)$ to $\ell^2(W)$.
The matrix entries satisfy inequalities
\begin{align*}
|(G-G_{\tau})_{k,j}| = \left| \frac{1}{k-j} - \frac{1}{k-j+\tau_k} \right| \leq \frac{1}{(k-j)^2}, \,\,\,\, k \neq j,
\end{align*}
so it suffices to show that any operator $A$ with matrix entries $A_{k,j}$,
\begin{align*}
A_{j,j}=0; \quad |A_{k,j}| \leq  \frac{1}{(k-j)^2}, \quad k \neq j
\end{align*}
is a bounded map from $\ell^2(W)$ to $\ell^2(W)$.
Indeed, decompose $A$ over its diagonals
\begin{align*}
A  = \sum_{\substack{t= -\infty\\ t \neq 0}}^\infty A^t \quad \text{with} 
\quad A^t_{i,i+j} = \delta(t,j) A_{i,i+t}, \quad  
\forall i \in \mathbb{Z}_{+}, \quad j \geq -i .
\end{align*}
So
$\| A^t \|_{2,W} \leq 
\max_{i \in \mathbb{Z}_+} \left( |A^t_{i,i+t}|
\left( |W(i+t)||w(i)| \right) \right)  \leq    2 t^{-2} r(t).$
Hence, by (\ref{25HA})
\begin{align*}
\|A \|_{2,W} \leq 2\sum_{\substack{t=-\infty \\ t\neq 0}}^\infty t^{-2 } r(t) < \infty
\end{align*}
and the lemma follows.
\end{proof}

\subsection{}
Define the space  
\begin{align*}
\ell^{2}(W, L^2(\mathbb{R})) &= \left\{ ( \xi_k(x) )_{k=0}^\infty: 
\xi_k(x) \in L^2(\mathbb{R}) \quad \forall k \in \mathbb{Z}_+ \,\, \text{and} \,\, 
\{ \|\xi_k(x)\|_2 \}_{k=0}^{\infty} \in \ell^2(W) \right\} 
\end{align*}
of $L^2(\mathbb{R})$-valued sequences
with the inner product
$\left< \xi, \eta \right> = \sum_{k=0}^{\infty} W(k) \left< \xi_k(x), \eta_k(x) \right>.$
\begin{lemma} \label{34H} 
Suppose $G_{\tau}$ is a bounded map from $\ell^2(W)$ into itself. The perturbed Hilbert transform $\widetilde{G}_\tau$ defined 
in the space of $L^2(\mathbb{R})$-valued sequences by
\begin{align*}
(\widetilde{G}_\tau (\xi_k(x)))_n = \sum_{\substack{j=0 \\ j\neq n}}^{\infty} \frac{\xi_k(x)}{k-n + \tau_k}
\end{align*}
is a bounded operator from $\ell^2(W, L^2(\mathbb{R}))$ into itself
with 
\begin{align} \label{36H} 
\|\widetilde{G}_{\tau} \|_{W} \leq  \|G_\tau \|_{W}.
\end{align}
\end{lemma}
\begin{proof}
Suppose $\xi =(\xi_j(x))_{j=0}^{\infty} \in \ell^2(W, L^2(\mathbb{R}))$ with
$
\xi_j(x) = \sum_{k=0}^{\infty} \xi_j^{(k)} h_k(x) \in L^2(\mathbb{R}).
$
We have
\begin{align*}
\| \widetilde{G}_\tau \xi \|_W^2 &= \sum_{n=0}^{\infty} \| (G_\tau \xi)_n \|_{L^2}^2 W(n) = 
                       \sum_{n=0}^{\infty} 
		       \left\| \sum_{\substack{j=0 \\ j \neq n}}^\infty 
		       \frac{\sum_{k=0}^{\infty} \xi_j^{(k)} h_k(x)}{j-n+\tau_j} \right\|_{L^2} W(n) \\
		       &=\sum_{n=0}^{\infty}  \sum_{k=0}^{\infty} 
		       \left| \sum_{\substack{j=0 \\ j\neq n}}^{\infty} \frac{\xi_{j}^{(k)}}{j-n+\tau_j} \right|^2 W(n)
		       =\sum_{k=0}^{\infty} \sum_{n=0}^{\infty} |(G_\tau \xi^{(k)})_n |^2 W(n)  \\ 
		       &=\sum_{k=0}^{\infty}   \| G_{\tau} \xi^{(k)} \|_W^2 \leq \|G_\tau\|_W^2 \sum_{k=0}^{\infty} \| \xi^{(k)} \|_W^2 = 
		        \|G_\tau \|_W^2 \| \xi \|_W^2.
\end{align*}
\end{proof}
Define the ellipsoid in $\ell_2(L^2(\mathbb{R}))$:
\begin{align*}
\widetilde{E}(W) := \left\{ \xi \in \ell^2(W, L^2(\mathbb{R})): \sum_{k=0}^{\infty} \left(\|\xi_k(x)\|_2^2 W(n) \right) \leq 1 \right\}.
\end{align*}
Under the conditions of Lemma \ref{25H} we have  
\begin{align} \label{40H} 
\widetilde{G}_{\tau}( \widetilde{E}(W) ) \subset C(W)  \widetilde{E}(W).
\end{align}
\begin{remark}
In \cite{mityagin2}, \cite{mityagin1} analysis of the spectra of $1$D periodic Dirac operators used in an essential 
way the discrete Hilbert transform to prove localization of $\text{Sp}(L)$ in small discs around $n \in \mathbb{Z}$ for
$|n|$ large enough. 
\end{remark}

\section{Proof of Propositions \ref{11A} and \ref{11AA}} \label{1C}

\begin{proof} 

Because $b \in V$ 
we may choose 
$J \in \mathbb{Z}_+$ 
with 
\begin{align}
\label{2C}
\| h_k b \|_2 \leq \frac{1}{68} \quad \text{whenever $k \geq J$}.
\end{align} 
Recall $\|b \| = \sup \{ \|h_k b \|_2 \}.$ Choose $N_* \in \mathbb{N}$ with
\begin{align} 
\label{3C}
N_* \geq  \frac{(2J+4 \|b\| \sqrt{J} + 1)}{2}  
\end{align}
so 
\begin{align} 
\label{4C}
 \frac{\|b\|^2 J}{(2(N_*-J) - 1)^2}  \leq   \frac{1}{16}.
\end{align}
Fix $z \notin S$, with $S=S(N_*)$ from (\ref{12A}). 
We will show $z \notin \text{Sp}(L)$.  
It is enough to show
$\|BR^0(z)\| \leq 1/2$ 
since then
\begin{align} 
\label{5C}
R(z) = R^{0}(z)(I-BR^0(z))^{-1}.
\end{align}
Let $f = \sum f_j h_j \in L^2(\mathbb{R})$
with $\| f \|_2 = 1$. 
We have, by (\ref{2C}) and Cauchy's inequality,
\begin{align*}
\|BR^{0}(z) f \|_2^2 & =  \left\| 
	\sum_{k=0}^{\infty} \frac{f_kb h_k}{z-\lambda^0_k}  
	\right\|_2^2\leq   \left( \sum_{k=0}^{\infty} 
	|f_k| \frac{\|bh_k \|_2}{|z-\lambda^0_k|} \right)^2  \\
	&  \leq 
		 \sup_{j < J} \{ \|h_j b \|_2^2 \} S_1 + 
		 \sup_{j \geq J} \{ \|h_j b \|_2^2 \} S_2 
\end{align*}
where
\begin{align*}
S_1 =\sum_{k=0}^{J-1} \left|z-\lambda^0_k \right|^{-2} \quad  
\text{and} \,\,\,\,
S_2 =\sum_{k=J}^{\infty} \left|z-\lambda^0_k \right|^{-2}.
\end{align*}
\\

If $| \text{Re}(z) | < 2N_*$ then 
$\text{Im}(z) \geq Y$. Thus,
\begin{align*}
S_1 &\leq \sum_{k=0}^{J-1} 
\left[ (\text{Re}(z)-\lambda^0_k)^2+Y^2 \right]^{-1} 
\leq Y^{-2}  \sum_{k \in \mathbb{Z}} 
\left[ 
\left(\frac{\text{Re}(z)-(2k+1)}{Y} \right)^2+1  
\right]^{-1}  \\
&\leq Y^{-2} \left(
2 + \int_{-\infty}^\infty 
\frac{1}{(2 x /Y)^2+1} dx \right) \leq   
Y^{-2}  \left(2+ \frac{Y \pi}{2}
\right).
\end{align*}
So (\ref{10A}) implies $\sup_{j < J} \{ \|h_j b \|_2^2 \} S_1 \leq 1/16$.
\\

Now suppose $| \text{Re}(z) | >2N_*$. We have
\begin{align*}
S_1 &= \sum_{k=0}^{J-1} 
\left[ (\text{Re}(z)-\lambda^0_k)^2+\text{Im}(z)^2 \right]^{-1} \leq 
\left( \frac{J}{(2(N_*-J)-1)^2} \right). 
\end{align*}
So (\ref{4C}) implies $\sup_{j < J} \{ \|h_j b \|_2^2 \} S_1 \leq 1/16$.
\vspace{.5cm}

Finally, because 
\begin{align*}
S_2 & = \sum_{j=J}^{\infty} 
\frac{1}{|z-\lambda_j^0|^2} \leq 
\left( 16^2 +1+  2 \sum_{j=1}^{\infty} 
\frac{1}{(2j)^2} \right) \leq 17^2,  
\end{align*}
the condition (\ref{2C}) implies $\sup_{j \geq J} \{ \|h_j b \|_2^2 \} S_2 \leq 1/16$.

By proving
\begin{align}\label{13C}
\|BR^0(z)\| \leq 1/2, \quad \text{for  $z \notin S$}
\end{align}
we have shown $\text{Sp}(L) \subset S$. So we have proven (\ref{12A}).

Consider the family
\begin{align*}
L(t) = L^0 + tB, \quad 0\leq t \leq 1.
\end{align*}
For $\xi \notin S$ we have  
\begin{align*}
R(L(t),\xi)=R^0(\xi) ( I - tBR^0(\xi))^{-1}  
\end{align*}
so (\ref{13C}) implies that $ R(L(t),\xi) $ 
is well-defined and depends continuously on $t \in [0,1].$
\\
Thus
\begin{align*}
S_*(t)   &:= \frac{1}{2 \pi i} 
\int_{\partial \Pi(2N_*,Y)} R(L(t),z) dz    \\
\text{and} \quad
P_k(t) &:= \frac{1}{2 \pi i} 
\int_{\partial D(\lambda^0_k,1)} R(L(t),z) dz, 
\quad \text{for $k\geq N_*$}   
\end{align*}
also depend continuously on $t \in [0,1]$. 
Since $\text{Sp}(L(0))$ contains only simple eigenvalues, it 
follows from Lemma $4.10$ in Chapter 1 of \cite{katobook} that
\begin{align*}
\text{dim}(S_*(t)) &= \text{dim}(S_*(0)) = N_* \\ 
\text{and} \quad \text{dim}(P_k(t)) &= 
\text{dim}(P_k(0)) = 1 \quad \text{for $k\geq N_*$}
\end{align*}
whenever $t \in [0,1]$.
\\
So, $\text{dim}(S_*) = \text{dim}(S_*(1)) = N_*, 
\quad \text{dim}(P_k)=
\text{dim}(P_k(1))=1, \quad k\geq N_*$ 
therefore the spectrum of $L =L(1)$ is discrete 
and contains exactly one (simple) eigenvalue in each 
$D(\lambda_k^0, 1/16)$ for each $k \geq N_*$.
Because $z \notin S$ implies 
$\|R^0(z)\| \leq 16$, (\ref{5C}) and (\ref{13C}) imply 
 (\ref{19A}), (\ref{20A}) and (\ref{21A}). 
\end{proof}

\section{Proof of Theorem \ref{22A}} \label{1D}

The following is a lemma from \cite{katopaper},  \cite[Ch5, Lemma 4.17a]{katobook}.
\begin{lemma} \label{2K}
Let $\{Q^0_k\}_{j \in \mathbb{Z}_+}$ be a complete family of orthogonal projections in a 
Hilbert space $X$ and let 
$\{Q_k\}_{j \in \mathbb{Z}_+}$ be a family of (not necessarily orthogonal) projections such that
$Q_j Q_k = \delta_{j,k}Q_j$. Assume that 
\begin{align}
\text{dim}(Q^0_0) = \text{dim}(Q_0) = m< \infty \\
\sum_{j=1}^\infty \| Q^0_j(Q_j-Q^0_j)u\|^2 \leq c_0 \|u\|^2, \quad \text{for every } \quad u \in X
\end{align}
where $c_0$ is a constant smaller than $1$. Then there is a bounded operator $W: X \rightarrow X$ with bounded inverse
such that $Q_j = W^{-1} Q^0_j W$ for $j \in \mathbb{Z}_+$.
\end{lemma}
\begin{remark}
The Bari-Markus criterion is often given with more restrictive conditions for norms of deviations
\begin{align*}
\sum_{j=0}^{\infty} \| Q_j - Q^0_j \|^2 < \infty
\end{align*}
and with an algebraic assumption
\begin{align*}
\text{dim}(Q_j) = \text{dim}(Q_j^0) < \infty, \quad j=0,1,\ldots
\end{align*}
-see \cite{markus}, \cite{markus2} and \cite{gohberg} Ch. 6, Sect. 5.3, Theorem 5.2.

These conditions have been proven for 
Sturm-Louiville operators on $I = [0,\pi]$ with a singular potential $v \in H^{-1}$ for periodic, antiperiodic or 
Dirichlet boundary value problems \cite{dm1}, \cite{savchuk} ,  
or for $1$-dimensional Dirac operators on $[0,\pi]$ - see \cite{dm2}.
\end{remark}
We now go directly to the proof of Theorem \ref{22A}.

\begin{proof}
We first show that there exists an 
integer $N \geq N_*$ with $N_*$ from  Proposition \ref{11A} such that 
\begin{align}\label{23}
\sum_{k \geq N_1} \left\| P_k^{0}(P_k-P_k^{0}) f\right\|_2^2 \leq \frac{1}{2}, \quad \text{for each } f \in L^2(\mathbb{R}); \quad \|f\|_2=1.
\end{align}
Suppose $n\geq N_*$. 
If we write $R$ and $R^0$ as
\begin{align*}
R^0(z) = \frac{P_n^0}{z-\lambda^0_n} + \Phi_n^0(z) \quad \text{and} \quad
R(z)   = \frac{P_n}{z - \lambda_n} + \Phi_n(z)
\end{align*}
then $\Phi_n$ and $\Phi_n^0$ are analytic, operator-valued functions in $D(\lambda^0_n, 1)$.
\\
Set 
\begin{align*}
                  \Psi_n^0 := \Phi_n^0(\lambda_n) \quad
\text{and} \quad \Psi_n   := \Phi_n(\lambda_n).
\end{align*}
From the identity 
\begin{align*}
P_n-P_n^0 &= \frac{1}{2\pi i} \int_{|z-\lambda^0_n|=1/4} (R(z) - R^0(z))dz \\ 
&=\frac{1}{2\pi i} \int_{|z-\lambda^0_n|=1/4} R(z) B R^0(z) dz 
\end{align*} 
we have 
\footnote{This representation is essentially from \cite[Eqn. 4.38, Ch. 5]{katobook}; but there the terms of positive degree from the 
Laurent expansions of $R(\xi,T)$ and $R(\xi,S)$ are not taken into account.}
\begin{align*}
P_n-P_n^0 &= P_n B \Psi_n^0 + \Psi_n B P^0_n. 
\end{align*}
Let $f = \sum_{j=0}^\infty f_j h_j(x) \in L^2(\mathbb{R})$ with $\|f \|_2=1$.\\
Then
\begin{align*}
\sum_{n\geq N} \| P^0_n(P^0_n - P_n)f \|_2^2 &= \sum_{n\geq N} \| P_n^0 P_n B \Psi_n^0 f + P_n^0 \Psi_n B P^0_n f \|_2^2  \\
	& \leq  2  \sum_{n \geq N} \left( \| P_n^0 P_n B \Psi_n^0 f\|_2^2 + \| P_n^0 \Psi_n B P^0_n f \|_2^2 \right). 
\end{align*}

Let $W_\psi$ be the sequence from (\ref{0303_3})  with $b \neq 0$
\begin{align} \label{0308_1}
\psi(k)= \left[ \sup \{\|h_j b\|_2: j\geq k\} \right]^{-1}
\end{align}
and let $C(W_{\psi})$ be from (\ref{40H}), (\ref{M}).
We will now show that if $N$ is chosen with 
\begin{align}\label{24}
N &\geq   \min \{ n: W_{\psi}(n)^{-1} \leq C(W_{\psi})/64 \}\\
\label{25}
\text{and }\quad N &\geq  \min \{ n: \| h_n b \|_2 \leq (70)^{-1} \}
\end{align}
then (\ref{23}) holds.

To come to this claim it suffices to show 
\begin{align}\label{26}
\sum_{n\geq N} \| P_n^0 P_n B \Psi_n^0 f\|_2^2 &\leq 1/4 
\end{align}
and
\begin{align}
\label{26ii} 
\sum_{n\geq N}\| P_n^0 \Psi_n B P^0_n f \|_2^2  &\leq 1/4.
\end{align}

\emph{Proof of (\ref{26})}:
\\
By (\ref{20A})
\begin{align*}
\sum_{n\geq N} \| P_n^0 P_n B \Psi_n^0 f\|_2^2 & \leq  32^2 \sum_{n \geq N} \|B \Psi_n^0f\|_2^2.
\end{align*}

Now,
\begin{align*} 
\sum_{n\geq N} \|B \Psi_n^0f\|_2^2 &= 
\sum_{n\geq N}   \left\| \sum_{k \neq n} 
\frac{ B P^0_k f}{\lambda_k^0 -\lambda_n}  \right\|_2^2     = 
 \sum_{n\geq N}   \left\| \sum_{k \neq n} 
 \frac{f_k b(x)h_k(x)}{\lambda_k^0 -\lambda_n}  \right\|_2^2.
\end{align*}
We have 
\begin{align}\label{N}
\sum_{n\geq N}   \left\| \sum_{k \neq n} 
 \frac{f_k b(x)h_k(x)}{\lambda_k^0 -\lambda_n}  \right\|_2^2 = 
 \| I_N \widetilde{G}_{\tau} \xi \|^2_2 
\end{align}
where
\begin{align*}
\xi &= (\xi_k(x))_{k=0}^{\infty}, \,\, \xi_k = f_k b(x) h_k(x), \\
\tau_j &= \lambda_j^0 - \lambda_j, \quad \text{and} \\
(I_j \eta)_k &= \begin{cases}  
\eta_k \quad \text{if} \quad k \geq j, \\
	0       \quad \text{if} \quad k < j
\end{cases}
\end{align*}
or - more generally - for $F \subset \mathbb{Z}_+$
\begin{align}
\label{0310_3}
(I(F) \eta)_k &= \begin{cases}  
\eta_k \quad \text{if} \quad k \in F, \\
	0       \quad \text{if} \quad k \notin F.
\end{cases}
\end{align}
If for $b$ we take $\psi_k$ as in (\ref{0308_1}), the vector valued sequence $\xi$ 
belongs to $\widetilde{E}(W_{\psi}, L^2(\mathbb{R}))$ -- see (\ref{0303_3}). So , by (\ref{40H})
\begin{align} \label{M}
\widetilde{G}_{\tau} \xi \in C(W_\psi)  \widetilde{E}(W_\psi).
\end{align}
By (\ref{24}), (\ref{N}), and (\ref{M})
\begin{align} \label{26a}
32^2 \sum_{n\geq N}   \left\| \sum_{k \neq n} 
 \frac{f_k b(x)h_k(x)}{\lambda_k^0 -\lambda_n}  \right\|_2^2 \leq 
 32^2 C(W_\psi)^2 W_\psi(N)^{-2} \leq  1/4.
\end{align}
So (\ref{26}) holds.
\\ \\
\emph{Proof of $(\ref{26ii})$}:
\\
By (\ref{21A}) and (\ref{25})
\begin{align*}
\sum_{n\geq N}\| P_n^0 \Psi_n B P^0_n f \|_2^2  &\leq \|\Psi_n \|^2 \sum_{n \geq N}\|   B P^0_n f \|_2^2    \\
&\leq 35^2 \sum_{n \geq N} |f_n|^2 \| b h_n \|_2^2 \leq 1/4.  
\end{align*}

By justifying (\ref{26}) and (\ref{26ii}) we have shown how to choose $N$ with
\begin{align} \label{26aa}
\sum_{k \geq N} \| P_k^{0}(P_k-P_k^{0}) f\|_2^2 \leq \frac{1}{2}\|f\|_2^2, \quad \forall f \in L^2(\mathbb{R}).
\end{align}
To complete the proof of Theorem \ref{22A}, it suffices to apply Lemma \ref{2K} to the orthogonal collection of projections
\begin{align*}
\{ S^0_{N}, P^0_{N},P^0_{N+1},\ldots \}
\end{align*}
where
\begin{align*}
S^0_{N} = P_0^0 + P_0^1 + \ldots + P^0_{N-1}
\end{align*}
and the (not necessarily orthogonal) collection of projections
\begin{align*}
\{ S_{N}, P_{N},P_{N+1},\ldots \} 
\end{align*}
where
\begin{align} \label{26b}
S_{N} = S_* + P_{N_*} + P_{N_{*}+1}+ \ldots + P_{N - 1}.
\end{align}
Inequality (\ref{26aa})  and Lemma \ref{2K} imply that the series 
\begin{align*}
f= S_{N} + \sum_{k\geq N} P_kf \quad \text{for all} \quad f \in L^2(\mathbb{R})
\end{align*}
converge unconditionally. By (\ref{26b}) it follows that  
the series (\ref{23A}) converge unconditionally as well.
\end{proof}


\section{Example spaces} \label{1E} 
The main hypothesis in  Propositions \ref{11A}-\ref{11AA} and Theorem \ref{22A} was $b \in V$.
In this section we will use known estimates for the Hermite functions $ \{ h_k \} $
to prove that the spaces (\ref{Oa}-\ref{Ob}) are embedded in $V$.
The following lemma is essentially \cite[Formula 8.91.10]{szego} and  \cite[Formula 6.11]{erdelyi} together with Theorem B and 
the table on p. 700 in \cite{askey}. See also \cite{skovgaard}.

\begin{lemma} \label{12}
Let $N=2n+1$. There are constants
\footnote{In what follows we use the letter  $C$ as a generic positive absolute constant.}
$C, \,\gamma >0$ such that

\begin{align} \label{13}
  |h_n(x)| \leq \begin{cases}  C (N^{1/3} + |x^2 - N| )^{-1/4}& \quad \text{if}\quad x^2 \leq 2N\\
                        C \text{exp}(-\gamma x^2)& \quad \text{if}\quad x^2 \geq 2N) \end{cases}
\end{align}
for all $n \in \mathbb{Z}_+$.
\end{lemma}

As Szeg\"o's book indicates, there is a long interesting history of estimates for Hermite, 
Laguerre and other orthogonal polynomials (see for example \cite{cramer},\cite{hille} and the paper of 
Erdelyi \cite{erdelyi}).
For later developments see \cite{nevai} and the references there. 
Inequalities (\ref{13}) imply that 
\begin{align} \label{0308_2}
|h_n(x)| \leq C(1+n)^{-1/12}, \quad x \in \mathbb{R}, n \in \mathbb{Z}_+.
\end{align}
An alternative proof of (\ref{0308_2}) is given in \cite[Lemma 4]{akhmerova}.

If $b \in L(p, \alpha)$  then
\begin{align} \label{0305_1}
\| b(x) h_n(x) \|_2 &= \| b(x)(1+|x|^2)^{-\alpha/2} h_n(x) (1+|x|^2)^{\alpha/2} \|_2 \\ 
\notag &\leq \| b(x) (1+|x|^2)^{-\alpha/2} \|_p \|h_n(x) (1+|x|^2)^{\alpha/2} \|_q
\end{align} 
where $1/p +1/q =1/2$.
Following the procedure outlined in \cite[Sect. 1.5, p.27]{thangavelu} 
(\ref{13}) can be used to bound
$$
\| h_n(x)(1+|x|^2)^{\alpha/2} \|_q = 2^{1/q} \left( \int_{0}^{\infty} |h_n(x)|^{q} (1+|x|^2)^{q \alpha/2}dx \right)^{1/q}.
$$
Breaking the integral into two parts and applying (\ref{13}) gives 
\begin{align} \label{0305_2}
\left[ 
\int_{0}^{\sqrt{N}/2}+\int_{3\sqrt{N}/2}^{2\sqrt{N}} +\int_{2\sqrt{N}}^{\infty} 
\right] |h_n(x)|^{q} (1+|x|^2)^{q \alpha/2} dx \leq 
C n^{q\alpha/2 - 1/2} 
\end{align}
and
\begin{align} \label{0305_3}
\int_{\sqrt{N}/2}^{3\sqrt{N}/2} |h_n(x)|^{q} (1+|x|^2)^{q \alpha/2} dx
\leq C n^{q(\alpha/2 - 1/12) - 1/6} \int_{0}^{N^{2/3}} (1+y)^{-q/4} dy.
\end{align}
We omit the details.
By (\ref{0305_1}) and (\ref{0305_2}-\ref{0305_3}) we have the following.
\begin{lemma} \label{0305_5}
Let $b \in L(p,\alpha)$ with $2 < p <\infty$ and $p \neq 4$. 
Then 
\begin{align} \label{0308_3}
\|b h_n\|_2 \leq C (n+1)^{\alpha/2 + t(p)}
\end{align}
with
\begin{align} \label{0305_4}
t(p) = \begin{cases} -\frac{1}{12}\left( 2-\frac{2}{p} \right)  & \quad \text{if}\quad  2 \leq p  < 4  \\
                     -\frac{1}{2p}& \quad \text{if}\quad 4 < p < \infty \end{cases}  
\end{align}
If $p=4$, then 
\begin{align} \label{0308_4}
\| b h_n \|_2 \leq C n^{\alpha/2 -1/8 }\log(n+2).
\end{align}
\end{lemma}
\begin{remark}
The estimates (\ref{0308_3}), (\ref{0308_4}) are sharp in the following sense:
\begin{align} \label{0308_5}
\|b h_n\|_2 \geq c (n+1)^{\alpha/2 + t(p)} \quad \text{if $2 \leq p < \infty$, $p \neq 4$}
\end{align}
and  
\begin{align} \label{0308_6}
\| b h_n \|_2 \geq c n^{\alpha/2 -1/8 }\log(n+2) \quad \text{if $p=4$.}
\end{align}
\end{remark}
\begin{proposition} \label{0312_1}
With $t$ as in (\ref{0305_4}), the spaces $L(p,\alpha)$ are embedded in 
$V$ whenever $\alpha/2 + t(p) \leq 0$ and $(p,\alpha) \neq (4,1/4)$.
Also, $L^\infty_0(\mathbb{R})$ is embedded in $V$.
\end{proposition}
\begin{proof}
If $\alpha/2 + t(p) < 0$ the result follows from Lemma \ref{0305_5}. 

Suppose now that $b \in Z,$ where $Z = L_0^\infty(\mathbb{R})$ or $Z = L(p,\alpha)$ with $\alpha/2 + t(p) =0$. 
In either case the map $\Phi:Z \rightarrow \ell^{\infty}(L^2(\mathbb{R}))$ 
defined by
\begin{align*}
\Phi(b) = \{ b(x) h_k(x) \}
\end{align*}
is bounded.
If 
$$
b \in Z_0 = \{\phi \in Z: \phi \quad \text{has compact support} \}
$$
then 
$$
\Phi(b) \in \ell_0^{\infty}(L^2(\mathbb{R})) = \{ \eta \in \ell^{\infty}(L^2(\mathbb{R})): \| \eta_k \|_2 \rightarrow 0 \}.
$$
But $Z_0$ is a dense subset of $Z$ and  
$\ell_0^{\infty}(L^2(\mathbb{R}))$ is a closed subspace of $\ell^{\infty}(L^2(\mathbb{R}))$ so 
we conclude that $\Phi(Z) \subset \ell_0^{\infty}(L^2(\mathbb{R}))$.
That is, $\| b h_k \|_2 \rightarrow 0$ whenever $b \in Z$. 
\end{proof}

\section{Bounded potentials} \label{NEWNEWONE1}
\subsection{}
Propositions \ref{11A}, \ref{11AA} and Theorem \ref{22A}
succeed in dealing with $L_0^\infty(\mathbb{R})$-potentials but the methods of the previous sections based 
on the property 
$
\| b h_k \| \rightarrow 0
$
of a  potential $b$ cannot be used to analyze an arbitrary $L^\infty(\mathbb{R})$-potential. 
For example, if 
\begin{align*}
m(x) = (-1)^nM,  \quad n \leq x < n+1, \quad n \in \mathbb{Z}
\end{align*}
then $\|m h_k \|_2 = M$ for all $k \geq 0$.
Of course, if $\| b | L^\infty \| = \rho < 1$
or $B$ is a bounded operator in $L^2(\mathbb{R})$ with $\| B \| = \rho <1$
the resolvent
\begin{align*}
R(z) = (I - R^0(z)B)^{-1}R^0(z) \\
     = R^0(z)(I-BR^0(z))
\end{align*}
is well-defined outside of the union of disks
\begin{align}
\label{0319_2}
U_\rho = \cup_{k=0}^{\infty} D(2k+1; \rho).
\end{align}
Indeed, for any $z \in \mathbb{C} \backslash U_\rho$
\begin{align*}
\|R^0(z) \| = \left( \min_{k} |z -(2k+1) | \right)^{-1} \\
	    = 1/r(z), \quad r(z) > \rho
\end{align*}
and
\begin{align*}
\|R^0(z) B \|, \|BR^0(z) \| \leq \rho/r(z) < 1.
\end{align*}
Therefore, the following is true.
\begin{proposition} \label{N1}
Let $B$ be a bounded operator in $L^2(\mathbb{R})$ with $\|B \| = \rho <1$. 
Then the operator
\begin{align} \label{4N}
L = -d^2 / dx^2 + x^2 + B
\end{align}
has a discrete spectrum $\text{Sp} L$; $\text{Sp}L \subset U_\rho \in (\ref{0319_2})$ and 
$\#\left( \text{Sp} L \cap D(2k+1; \rho) \right) = 1$ with a simple eigenvalue $\lambda_k = 2k+1 + \xi_k$, 
$|\xi_k| \leq \rho$, $k=0,1,2,\ldots$.
\end{proposition}
As in Section \ref{INTRO} Proposition \ref{11A}, (\ref{15A}), one could define projections 
\begin{align*}
P_n &= \frac{1}{2 \pi i} \int_{|z - (2n+1)| = r} R(z) dz, \quad \rho < r < 1 \quad \text{for all} \quad n=0,1,2,\ldots; \\
\text{dim}(P_n) &=1, \quad P_n f = \left< f, \psi_n \right> \phi_n.
\end{align*}
But now certainly there are no associated functions.

\subsection{}
The next step, i.e., the proof of an analog of Theorem \ref{22A}, is done by a direct 
application of (a special case with multiplicity $1$)  V. Katsnelson's theorem \cite[Thms. 2,3 ]{katsnelson}:
\begin{proposition} \label{215A}
Let $A$ be a dissipative operator in a Hilbert space $H$ with a discrete spectrum 
$\text{Sp} A = \{ \mu_n \}_0^\infty$, each $\mu_n$ being a simple eigenvalue: $A \phi_n =\mu_n \phi_n, \quad n \in \mathbb{Z}_+$.

If 
\begin{align} \label{0308_7}
\sup_{0 \leq j < \infty} \sum_{\substack{ k=0\\ k \neq j } } 
\frac{ \text{Im} \mu_j \text{Im} \mu_k }{|\mu_j - \bar{\mu_k}|^2} < \infty
\end{align}
and
\begin{align}
\sup_{\substack{0 \leq j,k < \infty\\ j\neq k} } 
\frac{4\text{Im} \mu_j \text{Im} \mu_k }{|\mu_j - \bar{\mu_k}|^2} < 1
\end{align}
then $\{ \phi_n \}_0^{\infty}$ is an unconditional basis in $H$, i.e., for some bounded
invertible operator $U: H \rightarrow H$ the system $e_n = U \phi_n, \quad n=0,1,\ldots$ is 
an orthogonal basis in $H$.
\end{proposition} 
It leads to the following.
\begin{theorem}[ ( \`{a} la folklore c. 1970) ] \label{NEWNEWONE2}
Let $B$ and $L$ be as in Proposition \ref{N1}, and 
\begin{align} \label{N2}
L \phi_k  = \lambda_k \phi_k, \quad \lambda_k \in D_k, \quad k=0,1,\ldots.
\end{align}
The system of eigenfunctions $\{ \phi_k \}_0^{\infty}$ is an unconditional 
basis in $L^2(\mathbb{R})$.
\end{theorem}
\begin{proof}
To have a dissipative operator we consider 
\begin{align*}
\widetilde{L} = L+ i \rho = -d^2/dx^2 + x^2 + (B+ i \rho)
\end{align*}
so if $B = B_1+ B_2$, $B_j$ selfadjoint, $j=1,2$, then 
$B_2 + \rho \geq 0$. The eigenvalues are shifted as well
so
\begin{align}
\label{N3}
\mu_k                                           &= \lambda_k + i \rho = 2k+1+ \xi_k + i \rho, \\
\notag
\text{and} \quad  \text{Im} \mu_k   &\geq 0, \quad | \xi_k | \leq \rho, \quad k=0,1, \ldots.
\end{align}
Therefore, 
$2k+1 - \rho \leq \text{Re}\mu_k = \text{Re}(\lambda_k) \leq 2k+1+ \rho$, and
\begin{align*}
0 \leq \text{Im} \mu_k = \text{Im} \lambda_k + \rho \leq 2 \rho.
\end{align*}
The condition 
\begin{align*} 
\sup_{k \neq j} \frac{ 4 \text{Im} \mu_k \cdot \text{Im} \mu_j}{|\mu_k - \bar{\mu_j}|^2} < 1
\end{align*}
holds as it follows from the following estimates.
This ratio is the largest when $k$ and $j$ are neighbors, say, $j=k+1$.

Let $z = 2k+ 1+ x + i y  \quad w = 2k+3+u+ iv$; $|x|, |u| \leq \rho < 1$ and $0 \leq y,v \leq h=2 \rho$.
Then
\begin{align*}
\zeta^2 &= \frac{4 y v}{(2+x-u)^2 + (y+v)^2} \leq \frac{(y+v)^2}{(2(1-\rho))^2+ (y+v)^2} \\ 
        &\leq \frac{(2h)^2}{4(1-\rho)^2+(2h)^2} = \frac{4 \rho^2}{(1-\rho)^2+4\rho^2} < 1.
\end{align*}

Another condition (\ref{0308_7}) to check is 
\begin{align*}
s^* = \sup_{k} \sum_{j \neq k} \frac{1}{|\mu_k - \bar{\mu_j}|^2} < \infty.
\end{align*}
By (\ref{N3}) 
\begin{align*}
| \mu_k - \bar{\mu_j}| &= | 2(k-j) + \xi_k - \bar{\xi}_j| \\
	&\geq 2|k-j|-2\rho \geq 2(1-\rho)|k-j|
\end{align*}
and 
\begin{align*}
s^* &\leq \frac{1}{4(1-\rho)^2} \cdot 2 \cdot \frac{\pi^2}{6} 
< \frac{1}{(1-\rho)^2} < \infty.
\end{align*}
Proposition \ref{215A} implies that $\{ \phi_k \}$ is an unconditional basis.
\end{proof}

\begin{remark}
For any bounded operator $B$, with $R^0(z), \quad \text{Im}z \neq 0$, 
being of the Schatten class $S_r, \quad r>1$, 
the products $R^0(z)B$, $BR^0(z)$ are in $S_r$ as well so we can use 
M. Keldysh's theorem \cite[Thm. 4.3]{markus} to claim: \\
\emph{the spectrum $\text{Sp}(L^0 + B)$ is discrete and its SEAF is complete}.
\\
But this SEAF is not necessarily an unconditional basis.
\end{remark}

\subsection{}
We will now construct examples of bounded operators $B$,\quad $\| B \|=1$, 
such that the perturbation (\ref{4N}) has a discrete spectrum, all points of $\text{Sp}(L^0+B)$ 
are simple eigenvalues and (\ref{N2}) holds, the system $\{ \phi_k \}$ is complete but it is 
\emph{not} a basis in $L^2(\mathbb{R})$.

Special 2-dimensional blocks play an important role in this construction.
For $0 < t < 1, \quad 0 < s = 1-k^2 < 1, \quad k>0$ we define a $2 \times 2$ matrix 
\begin{align} \label{N4}
b = \begin{bmatrix} 1-t&&t \\ && \\-st&&-1+t\end{bmatrix}
\end{align}
\begin{lemma} \label{N5}
As an operator in $\mathbb{C}^2$ with the Euclidean norm
\begin{align}\label{N6}
1-t \leq \|b\| \leq 1;
\end{align}
moreover,
\begin{align} \label{N7}
1-\frac{1}{2} tk^2 \leq \|b\| \leq 1-\frac{1}{2} t(1-t) k^2.
\end{align}
\begin{proof}
Of course, $\| b \|$ is larger than any of its entries
so $1-t \leq \| b \|$. Then
\begin{align*}
b = (1-t)
\begin{bmatrix}
1 & 0 \\ 0 & -1
\end{bmatrix}
+ t
\begin{bmatrix}
0 & 1 \\
-s & 0
\end{bmatrix}
\end{align*}
and  $\|b \| \leq 1-t + t =1$ so (\ref{N6}) holds.
To get more accurate estimates let us take $e= (1 \; 1)^{T}$, then  
\begin{align} \label{N8}
\|b \|^2 &\geq \frac{\|be \|^2}{\|e\|^2} = \frac{1}{2} \left( 1+ 1-2 tk^2 +t^2 k^4 \right) \\
      \notag   &1-tk^2+ \frac{1}{2} t^2 k^4 \geq \left( 1-\frac{1}{2}tk^2 \right)^2
\end{align}
On the other side, for $p = (x \; y)^{T}$,\quad $|x|^2 + |y|^2 \leq 1$
\begin{align}\label{N9}
\| b p \|^2 &= |(1-t)x + ty |^2 + |stx + (1-t)y|^2 \\
  \notag          &\leq (1-t)^2 (|x|^2+|y|^2) + t^2|y|^2+s^2t^2|x|^2+ 2t(1-t)(1+s)|x||y| \\
  \notag	    &\leq (1-t)^2 + t^2 + t(1-t)(2-k^2) \\
  \notag	    & = 1- t(1-t)k^2 \leq \left( 1-\frac{1}{2} t (1-t)k^2 \right)^2.
\end{align}
So (\ref{N8}) and (\ref{N9}) justify (\ref{N7}).
\end{proof}
\end{lemma}
\begin{lemma}
Let 
\begin{align}
\ell^0 = \begin{bmatrix}E & 0 \\ 0 & E+2 \end{bmatrix} \quad \text{and} \quad \ell &= \ell^0 +b.
\end{align}
Then 
\begin{align}
\ell = (E+1) + tc \quad \text{where} \quad c = \begin{bmatrix} -1 & 1 \\ -1+k^2 & 1 \end{bmatrix} 
\end{align}
and
\begin{align} \label{0308_8}
cg^{\pm} &= \pm k g^{\pm} , \quad g^{\pm} = (1 \quad 1\pm k)^T, \\
\notag    \ell g^{\pm} &= (E+1) \pm tk.
\end{align}
\end{lemma}
\begin{proof}
As soon as these formulas are written they could be checked by direct substitution.
\end{proof}
\begin{lemma} \label{NEWONE1}
The angle $\alpha$ between eigenvectors $g^+$ and $g^-$ from (\ref{0308_8}) is determined by the equation
\begin{align*}
(\cos \alpha )^2 = 1 - \frac{k^2}{1+ \frac{1}{4} k^4} > 1-k^2
\end{align*}
so $\sin \alpha < k$. With the basis decomposition 
\begin{align} \label{NEWONE2}
f = \Phi_0(f) u_0 + \Phi_1(f) u_1, \quad u_0 = \frac{g^+}{ \|g^+\|}, \quad u_1 = \frac{g^-}{\|g^- \|}.
\end{align}
in $\mathbb{C}^2$ we have 
\begin{align}\label{NEWNEWONE}
\| \Phi_0 \| = \| \Phi_1 \| = \frac{1}{\sin \alpha} > \frac{1}{k}.
\end{align}
\end{lemma}
\begin{proof}
Again, direct evaluation shows
\begin{align*}
(\cos \alpha )^2 &= \frac{(g^+, g^-)^2}{ \| g^+ \|^2 \|g^-\|^2} = \frac{(1+(1+k)(1-k))^2}{(1+(1+k)^2)(1+(1-k)^2)} \\
                 &= 1- \frac{k^2}{1 + \frac{1}{4} k^4} > 1-k^2 \quad \text{and} \quad \sin \alpha < k.
\end{align*}
(\ref{NEWNEWONE}) comes from straight-forward calculations.
\end{proof}

\subsection{}
The special 2-dimensional blocks of this subsection give us a diagonal representation of 
2-dimensional elements of a "bad" perturbation $B$, $\|B\|=1$.
\begin{proposition}\label{NEWNEWONE3}
Let $H = \ell^2(\mathbb{Z}_+)$ and $L^0 e_n = (2n+1)e_n, \quad n=0,1,2,\ldots$. Put
$B = \{ b(m) \}_0^{\infty}$, $b(m) \in (\ref{N4})$ with fixed $t, \quad 0< t< 1$, and 
$s = s(m) = 1-k^2(m), \quad k(m) = 2^{-m-1}$, i.e., 
\begin{align}
Be_{2m}   &= (1-t)e_{2m} + te_{2m+1} \\
\notag Be_{2m+1} &= -s(m) t e_{2m} - (1-t) e_{2m+1}, \quad m=0,1,2,\ldots.
\end{align}
Then $\|B \| = 1$; for $L = L^0 + B$
\begin{align*}
\text{Sp}(L) = \{ 4m+2 \pm t\cdot k(m), \,\, m\in \mathbb{Z}_+ \},
\end{align*}
each point $\lambda \in \text{Sp} L $ is a simple eigenvalue,
the system of eigenvectors $\{ \phi_m, \psi_m \}$ 
\begin{align} \label{NEWONE3}
\phi_m &= e_{2m} + (1+k(m) ) e_{2m+1}, \\
\notag \psi_m &= e_{2m} + (1-k(m) ) e_{2m+1}
\end{align}
is complete in $H$ but it is \emph{not} a basis in $H$.
\end{proposition}
Of course, the same is true for any $k(m), \quad 0 < k(m) < 1, \quad 
\lim_{m \rightarrow \infty} k(m) = 0$.
\begin{proof}
The proof comes directly from Lemmas \ref{N5}-\ref{NEWONE1}.
We have $\|B \| = \sup \| b(m) \|$ so
\begin{align*}
1 \geq \| B \| \geq \sup \left\{ 1- \frac{1}{2} t k^2(m): m \in  \mathbb{Z}_+ \right\}=1,
\end{align*}
and $\|B \|=1$.

The 2-dimensional subspaces 
\begin{align*}
Y_m = \text{span} \{ e_{2m}, e_{2m+1} \} = \text{span} \{ \phi_m, \psi_m \}
\end{align*}
are invariant for $B$ and $L = L^0 + B$. If $Q_m$ is an orthogonal projection onto $Y_m$ then 
\begin{align*}
f = \sum_{0}^{\infty} Q_m f, \quad \forall f \in \ell^2(\mathbb{Z}),
\end{align*}
and of course these series converge unconditionally.
But $Q_m f = \Phi_m(f) \phi_m + \Psi_m (f) \psi_m$,
$\Phi_m(g) = \Psi_m(g) = 0$ if $g \perp Y_m$,
and by (\ref{NEWNEWONE})
\begin{align*}
\| \Phi_m \| = \| \Psi_m \| \geq \frac{1}{k(m)} = 2^{m+1} \rightarrow \infty.
\end{align*}
So the system (\ref{NEWONE3}) is not a basis.
\end{proof}

The examples of Proposition \ref{NEWNEWONE3} show that the constant $1$
is sharp in the hypothesis $\|B\| = \rho < 1$ (see Proposition \ref{N1} and 
Theorem \ref{NEWNEWONE2} ); even a non-strict inequality would not be good. However we 
do not have counterexamples where $Bf = b(x)f(x), \quad b \in L^\infty(\mathbb{R}).$

\section{Appendix}
We introduced the weighted spaces $\ell^2(W)$ in Section \ref{1B} formula (\ref{0810_1}) .
In this appendix we will explain or prove statements and inequalities (\ref{0303_1} - \ref{0303_2}) of Section \ref{1H}.
The following sufficient (and necessary) conditions, ``condition $A_2$," for boundedness of $G$ in 
$\ell^2(W)$
in terms of the weight are given in \cite[Thm. 10]{hunt}. 
Recall the convention $w(k):= W(k)^{-1}$.
\begin{proposition} \label{4H}
Let $\{ W(k) \}_{k=0}^{\infty}$ be a positive sequence. Then $G: \ell^2(W) \rightarrow \ell^2(W)$ 
is a bounded operator if and only if
\begin{align} \label{5H}
&\sup_{k, n \geq 0 } \sigma^+ (k,n) \, \sigma^-(k,n) = M = M(W) < \infty  \\
\text{where} \quad &\sigma^{+}(k,n) = \frac{1}{1+n} \sum_{k}^{k+n} W(j), \,\,\,\, 
\sigma^{-}(k,n) = \frac{1}{1+n} \sum_{k}^{k+n} w(j). \notag
\end{align}
\end{proposition}
This proposition could be used to prove (\ref{0303_3}-\ref{0303_2}) but 
in this appendix
we will provide a self-contained (elementary) proof of all statements on discrete Hilbert 
transform (or its perturbations) used in Sections 2-4.

Let $T_0 = 0 < T_1 < \ldots$ be a sequence of integers such that
\begin{align} \label{0308_9}
t_k = T_k - T_{k-1}, k=1,2,\ldots
\end{align} 
and 
\begin{align} \label{0308_10}
t_{k+1} \geq R t_k\quad k =1,2,\ldots; R>2.
\end{align}

Define $W: \mathbb{Z}_+ \rightarrow \mathbb{R}_+$ a monotone increasing weight sequence by 
\begin{align} \label{0308_11}
W(j) = 2^k, \quad T_k \leq j < T_{k+1}, \quad k = 0,1,\ldots.
\end{align}

\begin{proposition} \label{0310_8}
Under the assumptions (\ref{0308_9}), (\ref{0308_10}), (\ref{0308_11}) $G$ is a bounded map from 
$\ell^2(W)$  into itself.
\end{proposition}
\begin{proof}
Notice that (\ref{0308_10}) implies 
\begin{align} \label{0308_12}
t_q/t_p \geq R^{q-p} \quad \text{if} \quad q \geq p
\end{align}
and
\begin{align} \label{0308_13}
T_k \geq t_k \geq R^{k-1} t_1 > R^{k-1}, k\geq 1.
\end{align}
Therefore 
\begin{align} \label{0308_14}
s= \sum_{j=0}^{\infty} W(j) (1+j)^{-2} < \infty.
\end{align}
Indeed, 
\begin{align*}
s&= \sum_{k=0}^{\infty} 2^{k} \sum_{j=T_k}^{T_{k+1}-1} \frac{1}{(1+j)^2} < 
 t_1 + \sum_{k=1}^{\infty} 2^k T_k^{-1}\\ &< t_1 + R \sum_1^{\infty} \left( 2/R \right)^k 
=t_1 + \frac{2R}{R-2}.
\end{align*}

It suffices to prove   
\begin{align} \label{app3}
\| G a \|_W^2 = \sum_{n=0}^{\infty} W(n) \left| \sum_{j=0}^{\infty} \frac{a_j}{j-n} \right|^2 < \infty
\end{align}
and 
\begin{align} \label{app4}
\sum_{n=1}^{\infty} W(n) \left| \sum_{j=1}^{\infty} \frac{a_j}{j-n} \right|^2  \leq C \sum_{k=0}^{\infty} W(k) |a_k|^2 
\end{align}
for some constant $C > 0$ independent of $a$ just for  $a = \{ a_k \}_{k=0}^{\infty}$ which have  finite support, say $a_k = 0$ if $k > K$ and $\sum W(j) |a_j|^2=1$.
For such $a$ (\ref{0308_14}) guarantees that all series in (\ref{app3}) and the left side of (\ref{app4}) 
converge absolutely. Indeed,
if $n \geq 2K$ then
\begin{align*}
\left| \sum_{\substack{j=0\\ j\neq n} } ^{\infty} \frac{a_j}{j-n} \right|^2 &= 
\sum_{j=0}^{K} \frac{1}{(n-j)^2} \leq \frac{4K}{n^2}
\end{align*}
and the left side in (\ref{app4}) 
\begin{align*}
\leq \sum_{0}^{2K} \ldots + 4K \left( \sum_{2K+1}^{\infty} \frac{W(n)}{n^2}\right)
\end{align*}
by (\ref{0308_14}).

We now prove (\ref{app4}).
\begin{align} \label{app7}
\| G a \|_W^2 = \sum_{N=1}^{\infty} \sum_{n \in J_N} W(n) \left|  \sum_{\substack{j \in J(N)\\ j \neq n}} \frac{a_j}{j-n} + \sum_{p=2}^{\infty} 
\sum_{j \in  J_{N \pm p} } \frac{a_j}{j-n}   \right|^2
\end{align}
where 
\begin{align*}
J_N = [T_N, T_{N+1}), N \geq 0,\,\, J_N= \emptyset, \,\,N<0 \quad \text{and} \quad J(N) = J_{N-1} \cup J_N \cup J_{N+1}.
\end{align*}
Let $\widetilde{R}$ be a constant satisfying 
\begin{align} \label{app8}
1< \widetilde{R} < R/2
\end{align}
(for example $\widetilde{R} = 1/2 + R/4$ satisfies (\ref{app8}) in since $R>2$ in (\ref{0308_10})).
Then by Cauchy's inequality:
\begin{align} \label{app9}
 \left| \sum_{\substack{j \in J(N)\\ j\neq n}} \frac{a_j}{j-n} + \sum_{p=2}^{\infty} \sum_{j \in J_{N \pm p}} \frac{a_j}{j-n} \right|^2
  \leq 
  \left( \frac{1}{\widetilde{R}-1} \right) \left[ \widetilde{R} \left| \sum_{\substack{j \in J(N)\\ j\neq n}} \frac{a_j}{j-n}\right|^2
  	+ \sum_{p=2}^{\infty} \widetilde{R}^p \left| \sum_{j \in J_{N \pm p}} \frac{a_j}{j-n}\right|^2 \right]
\end{align}
which is, by another application of Cauchy's inequality
\begin{align} \label{app10}
\leq  \left( 
\frac{1}{\widetilde{R}-1} 
\right) 
\left[ 
\widetilde{R} 
\left| 
\sum_{\substack{j \in J(N)\\ j\neq n}} \frac{a_j}{j-n} 
\right|^2
+ \sum_{p=2}^{\infty}\widetilde{R}^p 
\left( 
\sum_{k \in J_{N \pm p}} W(k) |a_k|^2
\right) 
\left( \sum_{j \in J_{N \pm p}} w(j) (j-n)^{-2} 
\right) 
\right].
\end{align}
Combining (\ref{app7}) with (\ref{app9}-\ref{app10}) we have 
\begin{align} \label{app11}
 \| G a \|_W^2 \leq S_0+ S_- + S_+
 \end{align}
where
\begin{align} \label{app12}
S_0&=\sum_{N=1}^{\infty} \sum_{\substack{n \in J_N \\ j\neq n}} W(n) \widetilde{R} \left| \sum_{j \in J(N)} \frac{a_j}{j-n} \right|^2
\\
\label{app13}
S_-&=\sum_{N=1}^{\infty} \sum_{p=2}^{\infty} 
\left( \sum_{k \in J_{N \pm p}} W(k) |a_k|^2 \right)
\widetilde{R}^p \sum_{n \in J_N}
W(n)
\left( \sum_{j \in J_{N - p}} w(j) (j-n)^{-2} \right)
\\
\label{app14}
S_+&=\sum_{N=1}^{\infty} \sum_{p=2}^{\infty} 
\left( \sum_{k \in J_{N \pm p}} W(k) |a_k|^2 \right)
\widetilde{R}^p \sum_{n \in J_N}
W(n)
\left( \sum_{j \in J_{N + p}} w(j) (j-n)^{-2} \right)
\end{align}

To bound $S_0 \in (\ref{app12})$ we use the fact that the cannonical Hilbert transform $G: \ell^2 \rightarrow \ell^2$ 
is bounded and all projections $I(F) \in (\ref{0310_3})$ are of the norm $1$  (in any $\ell^2(W)$).
\begin{align} \label{app15}
\sum_{N=1}^{\infty} &\sum_{n \in J_N} W(n)  \left| \sum_{\substack{j \in J(N) \\ j\neq n}} \frac{a_j}{j-n} \right|^2 = 
\sum_{N=1}^{\infty} \sum_{n \in J_N} 2^N  \left| G( I(J_N) a ) (n) \right|^2 \\ 
\notag &\leq \sum_{N=1}^{\infty} 2^N  \| G | \ell^2 \|^2 \sum_{j \in J_N}  |a_j|^2 =  \| G | \ell^2 \|^2\sum_{n=1}^{\infty}W(n) 
  |a_n |^2.
\end{align}
To bound $S_{-} \in (\ref{app13})$ 
\begin{align}
\sum_{N=1}^{\infty} &\sum_{p=2}^{\infty} 
\left( \sum_{k \in J_{N \pm p}} W(k) |a_k|^2 \right)
\widetilde{R}^p \sum_{n \in J_N}
W(n)
\left( \sum_{j \in J_{N - p}} w(j) (j-n)^{-2} \right)
\\
\notag &\leq 
\sum_{N=1}^{\infty} \sum_{p=2}^{\infty} 
\left( \sum_{k \in J_{N \pm p}} W(k)|a_k|^2 \right)
\widetilde{R}^p \sum_{n \in J_N}
\left( \sum_{j \in J_{N - p}}  2^{p} (j-n)^{-2} \right)
\\
\notag &\leq 
\sum_{N=1}^{\infty} \sum_{p=2}^{\infty} 
\left( \sum_{k \in J_{N \pm p}} W(k) |a_k|^2 \right)
\widetilde{R}^p 2^p \frac{1}{2} R^{-p}\\  
\notag &=\sum_{N=1}^{\infty} \sum_{p=2}^{\infty} \frac{1}{2} \left( 2 \widetilde{R}/R\right)^p \left( \sum_{k \in J_{N \pm p}} W(k)|a_k|^2\right)
\end{align}
with $(2 \widetilde{R}/R) < 1$ - see (\ref{app8}).
\\
Hence, summing over $p$ and $N$ yields
\begin{align} \label{app16}
\sum_{N=1}^{\infty} &\sum_{p=2}^{\infty} 
\left( \sum_{k \in J_{N \pm p}} W(k) |a_k|^2 \right)
\widetilde{R}^p \sum_{n \in J_N}
W(n)
\left( \sum_{j \in J_{N - p}} w(j) (j-n)^{-2} \right)
\\
\notag &\leq \left( \frac{2 \widetilde{R}}{R} \right)^2 \left( \frac{1}{1-2\widetilde{R}/R} \right) \sum_{k=1}^{\infty} W(k) |a_k|^2.
\end{align}
A similar argument applied to $S_+$ yields: 
\begin{align} \label{app17}
\sum_{N=1}^{\infty} &\sum_{p=2}^{\infty} 
\left( \sum_{k \in J_{N \pm p}} W(k) |a_k|^2 \right)
\widetilde{R}^p \sum_{n \in J_N}
W(n)
\left( \sum_{j \in J_{N + p}} w(j) (j-n)^{-2} \right)
\\
\notag &\leq \left( \frac{2 \widetilde{R}}{R} \right)^2 \left( \frac{1}{1-2\widetilde{R}/R} \right) \sum_{k=1}^{\infty} W(k) |a_k|^2.
\end{align}
So combining (\ref{app11}) with (\ref{app15}), (\ref{app16}), and (\ref{app17}) we have
\begin{align}
\|G a \|_W^2 \leq  \left[ \widetilde{R} \| G | \ell^2 \|^2 + 
2 \left( 2 \widetilde{R}/R\right)^2 \left(\frac{1}{1-2\widetilde{R}/R} \right)\right] \| a \|_W^2
\end{align}
\end{proof}
\begin{remark} \label{0322_1}
We can ease the hypothesis (\ref{0308_10})  and assume a weaker condition
\begin{align} \label{0312_2}
\lim \inf t_{k+1}/t_k = R>2.
\end{align}
The proof could be adjusted; we omit the details.
\end{remark}
\begin{corollary}\label{0304_1}
(a) \quad Inequality (\ref{0303_1}) holds. \\
(b) \quad Given any weight sequence $\psi(k) \rightarrow \infty$ there exists another weight sequence $W$ which satisfies 
(\ref{0303_1}) and (\ref{0303_2}).
\end{corollary}
\begin{proof}
(a) \quad For any $\alpha > 0$ define 
\begin{align} \label{0310_4}
T_0 = 0, \quad T_k = \left[ 2^{k/\alpha}\right] + k ,\quad k=1,2,\ldots
\end{align}
and 
\begin{align} \label{0310_5}
U(j) = 2^k, \quad T_k \leq j < T_{k+1}.
\end{align}
Then with $W(n) = (n+1)^{\alpha}, \quad n \in \mathbb{Z}_+$ the two sequences $U$ and $W$ are equivalent, i.e., for 
some constants $0<c(\alpha) , C(\alpha)$
\begin{align*} 
0 < c(\alpha) \leq U(j)w(j) \leq C(\alpha) < \infty, \forall j
\end{align*}
so $\ell^2(U) = \ell^2(W)$ and the norms (\ref{0810_1}) with weights $U$ and $W$ are equivalent.
With (\ref{0310_4}) for $t_k = T_k - T_{k-1}$ we have :
\begin{align} \label{0310_7}
\lim_{k \rightarrow \infty} t_{k+1}/t_k = 2^{1/\alpha}.
\end{align}
If $\alpha <1$, (\ref{0310_7}) implies that 
\begin{align*}
t_{k+1} \geq R t_k \quad \text{for} \quad k \geq K_\alpha, \quad R = \frac{1}{2} (2+2^{1/\alpha})>2.
\end{align*}
Now Proposition \ref{0310_8} and Remark \ref{0322_1} formula (\ref{0312_2}) imply (\ref{0303_1}).
\\
(b) \quad Choose $R>2$. Define 
\begin{align*}
T_0&=0, \quad T_2 = \min \{ t: \psi(t) \geq 2, t \geq 1 \}, \\ 
T_{k+1} &= \min \{ t+ T_k : \psi(t+T_k) \geq 2^{k+1}, \quad t \geq Rt_k \}, \quad k =1,2,\ldots
\end{align*}
Then $W \in (\ref{0308_11})$ is good by Proposition \ref{0310_8}, and (\ref{0303_1}),(\ref{0303_2}) hold.
\end{proof}
To support Lemma \ref{25H} and its use in Sections 2-3 we now prove the following.
\begin{lemma} \label{0312_3}
Under the assumptions (\ref{0308_9}),(\ref{0308_10}),(\ref{0308_11}) the condition (\ref{25HA}) holds, i.e.,  
\begin{align} \label{0310_8_}
\sum_{n=0}^{\infty} \frac{r(n)}{(1+n)^2} < \infty \quad \text{where} \quad r(n) = \sup \left\{ W(i+n)w(i): i\geq -n \right\}.
\end{align}
\end{lemma}
\begin{proof}
With given $k, n \geq 0$ define $p,q$ by $T_p \leq k < T_{p+1}$, $T_q \leq k+n < T_{q+1}$.
Then $q \geq p \geq 0$; if $q \leq p+1$ then $w(k)W(k+n) = 2^{-p} \cdot 2^{q} \leq 2.$.
\\
If $q \geq p+2$ by (\ref{0308_10}) and (\ref{0308_12})
\begin{align*}
n &> T_q - T_{p+1} = \sum_{p+2}^q t_k  \geq t_{p+2} \sum_{0}^{q-p-2} R^{j}\\
  &\geq t_2 \cdot \frac{R^{q-p-1}-1}{R-1} \geq \frac{1}{\beta} R^{q-p}
\end{align*}
where $\beta = R^2/t_2.$
\\
Therefore, with $\gamma = \log 2 / \log R <1$
\begin{align*} 
2^{q-p} = (R^{q-p})^{\gamma} \leq (\beta n)^{\gamma}
\end{align*}
and for any $k,n$, $w(k)W(k+n) = 2^{q-p} \leq 2+(\beta n)^{\gamma}$.
So (\ref{0310_8_}) holds.
\end{proof}
\begin{acknowledgements}
We are grateful to Petr Siegl; he brought our attention to a series of questions 
on PT-operators. In the course of preparation of this paper we enjoyed discussions and advice of our 
colleagues M. Agranovich, V. Katsnelson, M. Lacy,  P. Nevai, B. Pavlov, E. Shargorodsky, A. Shkalikov, R. Stanton, A. Volberg, M. Znojil;
many thanks to them.
\end{acknowledgements}

\end{document}